\documentclass[11pt]{amsart}
\usepackage{amssymb,amsmath,amsthm,amsfonts} 
\usepackage{enumerate}
\usepackage {latexsym}
\usepackage{graphicx}
\usepackage{epsfig}
\usepackage{bbm}
\usepackage{mathabx}
\include{srctex}
\usepackage{marginnote}
\usepackage[letterpaper, top=1.2in, bottom=1.2in, outer=1.1in, inner=1.1in, heightrounded, marginparwidth=2.5cm, marginparsep=2mm]{geometry}
  
\allowdisplaybreaks
\newcommand{\nc}{\newcommand}
\nc{\les}{\lesssim}
\nc{\nit}{\noindent}
\nc{\nn}{\nonumber}
\nc{\D}{\partial}
\nc{\diff}[2]{\frac{d #1}{d #2}}
\nc{\diffn}[3]{\frac{d^{#3} #1}{d {#2}^{#3}}}
\nc{\pdiff}[2]{\frac{\partial #1}{\partial #2}}
\nc{\pdiffn}[3]{\frac{\partial^{#3} #1}{\partial{#2}^{#3}}}
\nc{\abs}[1] {\lvert #1 \rvert}
\nc{\cAc}{{\cal A}_c}
\nc{\cE}{{\cal E}}
\nc{\cF}{{\cal F}}
\nc{\cP}{{\cal P}}
\nc{\cV}{{\cal V}}
\nc{\cQ}{{\cal Q}}
\nc{\cGin}{{\cal G}_{\rm in}}
\nc{\cGout}{{\cal G}_{\rm out}}
\nc{\cO}{{\cal O}}
\nc{\Lav}{{\cal L}_{\rm av}}
\nc{\cL}{{\cal L}}
\nc{\cB}{{\cal B}}
\nc{\cZ}{{\cal Z}}
\nc{\mR}{{\mathcal R}}
\nc{\mG}{{\mathcal G}}
\nc{\cT}{{\cal T}}
\nc{\cY}{{\cal Y}}
\nc{\cX}{{\cal X}}
\nc{\cXT}{{{\cal X}(T)}}
\nc{\cBT}{{{\cal B}(T)}}
\nc{\vD}{{\vec \mathcal{D}}}
\nc{\efield}{\mathcal{E}}
\nc{\vE}{{\vec \efield}}
\nc{\vB}{{\vec \mathcal{B}}}
\nc{\vH}{{\vec \mathcal{H}}}
\nc{\F}{  \mathcal{F} }
\nc{\ty}{{\tilde y}}
\nc{\tu}{{\tilde u}}
\nc{\tV}{{\tilde V}}
\nc{\Pc}{{\bf P_c}}
\nc{\bx}{{\bf x}}
\nc{\bX}{{\bf X}}
\nc{\bXYZ}{{\bf XYZ}}
\nc{\bY}{{\bf Y}}
\nc{\bF}{{\bf F}}
\nc{\bS}{{\bf S}}
\nc{\dV}{{\delta V}}
\nc{\dE}{{\delta E}}
\nc{\TT}{{\Theta}}
\nc{\dPsi}{{\delta\Psi}}
\nc{\order}{{\cal O}}
\nc{\Rout}{R_{\rm out}}
\nc{\eplus}{e_+}
\nc{\eminus}{e_-}
\nc{\epm}{e_\pm}
\nc{\eps}{\varepsilon}
\nc{\vnabla}{{\vec\nabla}}
\nc{\G}{\Gamma}
\nc{\w}{\omega}
\nc{\mh}{h}
\nc{\mg}{g}
\nc{\vphi}{\varphi}
\nc{\tlambda}{\tilde\lambda}
\nc{\be}{\begin{equation}}
\nc{\ee}{\end{equation}}
\nc{\ba}{\begin{eqnarray}}
\nc{\ea}{\end{eqnarray}}

\nc{\g}{\gamma}
\nc{\ol}{\overline}

\newtheorem{theorem}{Theorem}[section]
\newtheorem{lemma}[theorem]{Lemma}
\newtheorem{prop}[theorem]{Proposition}
\newtheorem{corollary}[theorem]{Corollary}

\newtheorem{rmk}[theorem]{Remark}

\nc{\pT}{\partial_T}
\nc{\pz}{\partial_z}
\nc{\pt}{\partial_t}
\nc{\la}{\langle}
\nc{\ra}{\rangle}
\nc{\infint}{\int_{-\infty}^{\infty}}
\nc{\halfwidth}{6.5cm}
\nc{\figwidth}{10cm}
\newcommand{\f}{\frac}

\nc{\nlayers}{L} \nc{\nsectors}{M}
\nc{\indicator}{\mathbf{1}}
\nc{\chile}{R_{\rm hole}}
\nc{\Rring}{R_{\rm ring}}
\nc{\neff}{n_{\rm eff}}
\nc{\Frem}{F_{\rm rem}}
\nc{\R}{\mathbb R}
\nc{\C}{\mathbb C}
\nc{\Z}{\mathbb Z}
\nc{\DD}{\Delta}
\nc{\cD}{\mathcal D}
\nc{\lnorm}{\left\|}
\nc{\rnorm}{\right\|}
\nc{\rnormp}{\right\|_{\ell^{p,\eps}}}
\nc{\rar}{\rightarrow}
\begin{document}

\begin{abstract}
We study the scattering poles of Einstein–Schrödinger operator  $\sqrt{-\Delta} + V$, where $V$ is compactly supported, bounded and complex valued potential. We show that the resolvent operator $ \chi R_V \chi$ has meromorphic continuation to the whole Riemannian surface of $\Lambda$ of $ \log z $ as an operator  $L^2 \rightarrow L^2 $. We then obtain the upper bound on the counting function $N(r,a)= \# \{ z_j \in \Lambda: 0 \leq  |z_j| \leq r, |\arg z_j| \leq a \}$, $r >1$, $ |a| >1$  as $C   \la a \ra  ( \la r \ra^{d} + (\log \la a \ra)^d) $, where $z_j$ are the poles of $ \chi R_V(\lambda) \chi$. 
	
\end{abstract}
\title[Scattering poles of half-Laplacian]{Bounds on the number of scattering poles of half-Laplacian in odd dimensions, $d\geq 3$ }

\author[ Ebru Toprak]{ Ebru Toprak}


\maketitle

\section{Introduction}

  Scattering resonances are known to be generalization of the eigenvalues or bound states for systems in which energy can scatter to infinity, and they  are related to the poles of the meromorphic continuation of the resolvent operator. More precisely, let $V \in L^{\infty}(\mathbb{C}^d)$ be a compactly supported potential and let $R_H (\lambda) := (-\Delta +V -\lambda^2)^{-1}$ for $ \Im \lambda >0$, then it is known that $R_H (\lambda)$  is a bounded operator on $L^2 $, except for a finite number of values of $\lambda$. If $\chi$ is a compactly supported cut-off function such that $ \chi V =V$, then $ \chi R_H(\lambda) \chi$ has a meromorphic continuation to the lower half plane when $d$ is odd; and to $\Lambda$, the logarithmic cover of the complex plane, when d is even. The poles of  $ \chi R_H(\lambda) \chi$ are called resonances, and are independent of the choice of the cut-off function $\chi$. For a very detailed explanation of the scattering resonances we refer the reader to e.g. \cite{DZ}.

  The upper bounds on the number of  poles of $ \chi R_H(\lambda) \chi$ have been of great interest to many authors, see \cite{DZ} and references therein.  In fact, they have been studied in every dimension for a large class of compactly supported perturbations. Let $\{z_j\}$ be the poles of $ \chi R_H(\lambda) \chi$ repeated according to multiplicity and  let $ N(r) = \# \{ z_j : |z_j| \leq r \}$ when $d \geq3$ is odd. Then one has $N(r) \leq C_1 r^d$, see \cite{ SZ,Vod2}.
  
   The formulation for the counting function in even dimensions is more delicate as $R_H (\lambda)$ exhibits logarithmic singularity around zero.  One well known formulation considers $N(r,a)= \# \{ z_j \in \Lambda: 0 < |z_j| \leq r, |\arg z_j | \leq a\}$, $r,a >1$.  In \cite{VD}, Vodev showed that $N(r,a) \leq C a( r^d +( \log a)^d) $. Another counting function is considered by Intissar in \cite{Int} as $\tilde{N}(r) = \# \{ \lambda_j: r^{-\epsilon} \leq |\lambda_j| \leq \epsilon \log r\}$, $r>1$, $ \epsilon \in ( 0, 2^{-1/2})$ with the bound $ \tilde{N}(r)  \leq C r^{d+1}$. For the estimate in  dimension two one can see \cite{Vodtwo}.

  In this note, we study the scattering poles of  $H :=  \sqrt{-\Delta} + V $, where $V$ is bounded, compactly supported, and  complex valued. Recall that for $ s>0$, $ (-\Delta)^{s/2}$ is defined through the Fourier symbol $|\xi|^{s}$,  it is self-adjoint in $L^2(\R^n)$, with domain $H^{s}$ and it has absolutely continuous spectrum in $[0,\infty)$. It also has the following representation for any $ 0<s<2$ 
  \begin{align} \label{fracdef}
  (-\Delta)^{s/2} \psi(x) = \frac{2^s \Gamma( \frac{d+s}{2})}{ \pi^{d/2} |\Gamma( -\frac{s}{2})|} \text{p.v.} \int_{\R^d} \frac{\psi(x)- \psi(y)}{ |x-y|^{d+s}} dy. 
   \end{align}
  
 Our motivation is the recent paper \cite{ZHZ} on the wave operator related to $H$ . Let $W{\pm} = s-\lim_{t\to \pm \infty} e^{itH} e^{-it \sqrt{-\Delta}}$ be the aforementioned wave operator. It was shown in \cite{ZHZ} that $W{\pm} $ exist and is complete for short range potentials. Moreover, the scattering matrix $S := (W^{+})^* W_-$ is unitary. In Appendix~A, we also compute the scattering matrix for $H$ in odd dimensions when  $d \geq 3$, and  $V$ is real and compactly supported, similarly what is computed for the Schr\"odinger operator. In fact, using \eqref{fracdef} one can see that all solutions to  $ (\sqrt{-\Delta} - \lambda )  \psi_0=0$ for  $ \lambda > 0$  are the superposition of the elementary plane wave solutions $e^{ \pm i \lambda \la x , \omega \ra }$  where $ x \in \R^d$, and $ \omega \in \mathbb{S}^{d-1}$. Therefore, the absolute scattering matrix corresponding to $\sqrt{-\Delta}$  arises similar to the one corresponding to $-\Delta$. 
  
 It has to be mentioned that in this paper we do not answer the question whether or not $(H -  \lambda )u=0$ for $ \lambda >0 $ has outgoing solutions. In fact, to compute the scattering matrix we assume the absence of positive eigenvalues for $H$. It is not hard to see in a similar way to Agmon's bootstrapping argument  that under some decaying assumptions on $V$ the eigenfunctions of $H$ corresponding to the positive eigenvalues decay faster than any polynomial. On the other hand, there are very few papers which study the absence of embedded eigenvalues for $H$ and they either assume negativity or smallness on $V$, \cite{LS,RU}. This issue will be addressed in elsewhere.

   In this paper, we define  $R_0(z) := ( \sqrt{-\Delta} -z)^{-1}$, and $R_V(z) := ( \sqrt{-\Delta} +V - z)^{-1}$ for $\Im z > 0$, $\Re z> 0$. In Section~\ref{mcont}, we compute an expansion for $R_0(z)$  in $d \geq 3$ when d is odd. To the best of our knowledge such expansion did not appear in elsewhere.  As the expansion of  $R_0(z)$ displays logarithmic behaviour, the meromorphic  extension of $ \chi R_V(z) \chi $ is to $ \Lambda$. This also suggests to  study the poles of  $ \chi R_V(z) \chi $ in $N(r,a)$ as they are studied in  \cite{VD}. Our main theorem is the following
   
\begin{theorem}\label{the:main} Let $ V(x)$ be bounded, compactly supported, and  complex valued  function, then the counting function $N(r,a)$ of the scattering poles associated to the operator $ \sqrt{-\Delta} + V$ satisfies the bound 
$$ N(r,a) \leq C \la a \ra  ( \la r \ra^{d} + (\log \la a \ra)^d )  $$ 
where  $|a| \geq \frac{\pi}{4}$, $r \geq 1$. 
\end{theorem} 

    We end this section with notations used throughout this paper.  We use $L^2_c$ to represent compactly supported $L^2$ functions and $ L^2_{\text{loc}}$ to represent the locally $L^2$ functions. We use $ L^{2, \sigma} := \{ f: \|\la x\ra^{\sigma}f \|_{L^2} < \infty\}$, $\mathcal{H}_{s}$ for Sobolev norm, and $ \mathcal{H}_{s, \sigma}:=\{ f: \|\la x\ra^{\sigma}f \|_{\mathcal{H}_s} < \infty\}$.

\section{Meromorphic continuation of $ \chi R_V \chi $}\label{mcont}
Let $V$ be bounded compactly supported, and complex valued function. Define the resolvent operator 
\begin{align}
R_0(z):=( \sqrt{-\Delta} - z)^{-1},  \,\,\,\,\,\, R_V(z):=( \sqrt{-\Delta} + V- z)^{-1} \,\,\,\,\,\,\text{for} \,\,\,\  \Im z>0, \,\,\, \Re z>0  . 
\end{align}
We represent the integral kernel of $R_0(z)$ as $R_0(z)(x,y)$. In this section we will prove that 
 \begin{prop} \label{prop:mer} Let $ \chi(x) \in C_c^{\infty}$ be a cutoff function such that $ \chi V = V$. Then 
 $ \chi R_V(z) \chi $ extends meromorphically to the whole Riemannian surface of $\Lambda$ of $\log z $ as an operator $ L^2 \rightarrow L^2$. The poles of the extension of $ \chi R_V \chi $ coincides with the poles of the extension of  $ ( I + \chi R_0 V)^{-1} $, and have finite multiplicity.
  \end{prop} 
 
 Before we prove Proposition~\ref{prop:mer}, we need to show that $ \chi R_0(z) \chi $  has an analytic extension.
 
  \begin{lemma} \label{lem:mapping} Let $d \geq 3$ be odd, and $\chi$ be a compactly supported cut-off function. $ \chi R_0(z) \chi$  extends analytically to the whole Riemannian surface of $\Lambda$ of $\log z $ as an operator $ L^2 \rightarrow L^2 $ and  one has 
\begin{align} \label{L2bound}
\| \chi R_0(z) \chi \|_{ L^2 \rightarrow L^2} \les  e^{|z|}. 
\end{align} 
\end{lemma}

\begin{proof} We first show that for any $ z $ with $\Im z >0, \Re z > 0$,  one has the following expansion 
\begin{align}\label{R0main}
 R_0(z) (x,y)  = \sum_{k=0}^{\frac{d-3}{2}}  \frac{\tilde{c_k}}{ |x-y|^{d-1-s_k}}  \partial_s^{s_k} \Big\{  e^{izs}  E_i( - izs)+e^{-izs} E_i ( -izs) +   2 \pi i e^{iz s} \Big\}\Big|_{(s=|x-y|)} 
\end{align}
where $\tilde{c_k}$'s are  complex coefficients, $s_k= \frac{d-1}{2} -k$, and 
\begin{align} \label{E1}
E_i(\sigma) =  \gamma  + \ln (-\sigma)  + \sum_{k=1}^{\infty}\frac{ \sigma ^k}{ k k!}.
\end{align}
Here $ \gamma$ is the Euler–Mascheroni constant, and  the branch cut  for $\log \sigma $ is  $[0, \infty)$.  

Define 
\begin{align} \label{Rep}
 R^{\epsilon}_0(z) (x,y) := \frac{1}{ ( 2 \pi)^d} \int_{\R^d} e^{i  (x-y) \cdot \xi } \frac{ e^{- \epsilon|\xi|}}{ |\xi| - z} d\xi.
\end{align}
For any  $f \in L^2(\R^d)$ and $\Im z >0, \Re z > 0$, $R^{\epsilon}_0(z) (x,y)$ converges to $R_0(x,y)$  pointwise. Moreover, one has 
\begin{align*}
\| R_0(z) \ast f\|_{L^2} = \Big\| \frac{ \hat{f}(\xi)}{ |\xi|-z} \Big\|_{L^2_\xi} \les \|f\|_{L^2}
\end{align*}
Therefore, by dominated convergence theorem $R^{\epsilon}_0(z)$ converges to $( \sqrt{-\Delta} -z)^{-1} $ as an $L^2 \to L^2$ operator.

Hence, we start computing for $\Im z >0, \Re z > 0$,
\begin{align}\label{epsilon}
 R^{\epsilon}_0(x,0)(z) & =\frac{1}{ ( 2 \pi)^d} \int_0^{\infty} \int_{\mathbb{S}^{d-1}}  e^{i r (\omega \cdot x) } \frac{ e^{- \epsilon r }}{ r- z} r^{d-1} d \omega dr  \\
 & =\frac{1}{ ( 2 \pi)^d} \sum_{k=0}^{\frac{d-3}{2}}   \frac{c_k}{  |x|^{\frac{d-1}{2}+k}} \Big(  \int_0^{\infty} \frac{e^{ i r |x| - \epsilon r }}{ r-z }  r^{\frac{d-1}{2}-k} dr + \int_0^{\infty} \frac{e^{ -i r |x| - \epsilon r }}{ r-z }  (-r)^{\frac{d-1}{2}-k} dr\Big) \nn 
  \end{align} 
where $d\omega$ denotes the measure on $ \mathbb{S}^{d-1}$, and $c_k$'s are complex coefficients. In the last equality we used the following expansion for any odd $d \geq 3$. 
\begin{align} \label{osc}
 \int_{\mathbb{S}^{d-1} } e^{ir \omega \cdot x } d\omega  &=  \sum_{k=0}^{\frac{d-3}{2}}  c_k \frac{e^{ i r |x|}}{ ( r|x|)^{\frac{d-1}{2}+k} }  + \sum_{k=0}^{\frac{d-3}{2}}  c_k \frac{e^{- i r |x|}}{ (- r|x|)^{\frac{d-1}{2}+k} } 
 \end{align}
 By Cauchy integral formula, we have 
 \begin{multline}
\int_0^{R} \frac{e^{ i r |x| - \epsilon r }}{ r-z } r^{\frac{d-1}{2}-k} dr + \int_{ i R}^0 \frac{e^{ i r |x| - \epsilon r }}{ r-z } r^{\frac{d-1}{2}-k} dr \\  + ( i)^{\frac{d-1}{2}-k} \int_{0}^{ \frac{\pi}{2}} \frac{e^{ i R e^{i\theta} |x| - \epsilon (R e^{i \theta}) }} {R e^{i \theta}- z }  ( R e^{i \theta})^{\frac{d+1}{2}-k} dr =  2 \pi i  e^{ i z |x| - \epsilon z} z^{\frac{d-1}{2}-k} 
\end{multline}
Similarly, 
\begin{multline}
\int_0^{R} \frac{e^{- i r |x| - \epsilon r }}{ r-z } (- r)^{\frac{d-1}{2}-k} dr + \int_{- i R}^0 \frac{e^{- i r |x| - \epsilon r }}{ r-z } (- r)^{\frac{d-1}{2}-k} dr \\  (- i)^{\frac{d-1}{2}-k} \int_{0}^{- \frac{\pi}{2}} \frac{e^{- i R e^{i\theta} |x| - \epsilon (R e^{i \theta}) }} {R e^{i \theta}- z }  ( R e^{i \theta})^{\frac{d+1}{2}-k} dr =0.
\end{multline}
Note that $ \epsilon >0$, and $|e^{\pm i R e^{i\theta} |x| - \epsilon (R e^{i \theta})}| \les e^{ R(\mp \sin \theta |x| - \epsilon \cos \theta)} $. Therefore, as $ R \to \infty$ we obtain 
\begin{multline}
\int_0^{\infty} \frac{e^{ i r |x| - \epsilon r }}{ r-z } ( r)^{\frac{d-1}{2}-k} dr = - \int_{ i \infty}^0 \frac{e^{ i r |x| -\epsilon r} }{ r-z } ( r)^{\frac{d-1}{2}-k} dr  +\frac{1}{ 2 \pi i} e^{ i z |x| - \epsilon z} z^{\frac{d-1}{2}-k} \\ = - \int_{-\infty}^{0} \frac{e^{r|x| }e^{ i \epsilon r } }{ r - i z} ( -ir)^{\frac{d-1}{2}-k} dr + 2 \pi i e^{ i z |x| - \epsilon z} z^{\frac{d-1}{2}-k} 
\end{multline} 
and 
\begin{align}
\int_0^{\infty} \frac{e^{- i r |x| - \epsilon r }}{ r-z } (- r)^{\frac{d-1}{2}-k} dr = - \int_{-\infty}^{0} \frac{e^{r|x| }e^{- i \epsilon r } }{ r + i z} ( -ir)^{\frac{d-1}{2}-k} dr.
\end{align} 
One has $s_k: =\frac{d-1}{2}-k \geq 1$, also $\Re z > 0$, therefore by dominated convergence theorem 
\begin{align}
\lim_{\epsilon \to 0} \int_{-\infty}^{0} \frac{e^{r|x| } e^{\pm i \epsilon r}  }{ r \mp i z} ( -ir)^{s_k} dr = \Big( (-i \partial_s)^{s_k}  \int_{-\infty}^{0} \frac{e^{rs} }{ r \mp i z} dr \Big)\Big|_ {s=|x|}
\end{align}

Let $r (t) := \sigma +t$ for any $\sigma \notin \R^{\geq 0}$. Then, the integral on the right hand side of the above equality can be computed using the following expansion, see \cite{AS}. 
\begin{align} \label{E1int}
E_i(\sigma):= \int_{r(-\infty)}^{r(0)}  \frac{e^{t}}{{t}} dt =  \gamma + \ln (-\sigma)  + \sum_{k=1}^{\infty}\frac{ \sigma^k}{ k k!}
\end{align}
  as 
\begin{align} 
 \int_{-\infty}^{0} \frac{e^{rs} }{ r \mp i z} dr   = e^{\pm izs} E_i (\mp izs). 
\end{align} 
Therefore, we obtain 
\begin{align} \label{highR0}
\lim_{\epsilon \to 0}  R^{\epsilon}_0(x,0)(z) = \frac{1}{ ( 2 \pi)^d} \sum_{k=0}^{\frac{d-3}{2}}  \frac{ c_k (-i)^{s_k}}{ |x|^{d-1-s_k}} \partial_s^{s_k} \Big\{ e^{izs} E_i(- izs) + e^{-izs} E_i( izs)  + 2 \pi i e^{iz s} \Big\}\Big|_{(s=|x|)}. 
\end{align} 
 which gives  \eqref{R0main}. 

To extend  $ \chi R_0(z) \chi$ to $\Lambda$, we take  the branch cut $[0,\infty)$ for $\log(z)$, and glue together $g_m(z) = \log|z| + i ( 2 \pi m + \arg(z)) $ where $ 0 < \arg z < 2 \pi $. We use  $\tilde{E}_i(\pm izs )$ for the analytic extension of $E_i(\pm izs )$ to $\Lambda$, and represent the extended resolvent as 
\begin{multline}\label{R0mainext}
R_0(z) (x,y) \\ = \frac{1}{ ( 2 \pi)^d}  \sum_{k=0}^{\frac{d-3}{2}} \frac{\tilde{c_k}}{ |x-y|^{d-1-s_k}}  \partial_s^{s_k} \Big\{  e^{izs} \tilde{E}_i( - izs)+e^{izs} \tilde{E}_i ( -izs) +  2 \pi i e^{iz s} \Big\}\Big|_{(s=|x-y|)} 
\end{multline}
for all $z \in \Lambda$. 

  Finally, the exponential bound \eqref{L2bound} arises naturally from the expansion \eqref{R0mainext}, and the fact that convolution with $|x|^{-s_k}$ for $ 0< s_k < d$ maps $L^2_{\text{c}}$ to $L^2_{\text{loc}}$, see, e.g., \cite{Jen}.
    \end{proof}
    
 \begin{rmk} \label{rmk}
 
 \begin{itemize} 
 \item[i)] If one  computes $ (\sqrt{-\Delta}- z)^{-1} $ for $ \Im z <0$, $ \Re z >0$, then one would obtain the following resolvent operator. 
 \begin{align}\label{R0in}
 R^{-}_0(z) (x,y)  =  \sum_{k=0}^{\frac{d-3}{2}}  \frac{\tilde{c_k}}{ |x-y|^{d-1-s_k}} \partial_s^{s_k} \Big\{  e^{izs} \tilde{E}_i(  -izs)+e^{-izs} \tilde{E}_i ( izs) + 2 \pi i  e^{-iz s} \Big\}\Big|_{(s=|x-y|)} 
\end{align}

\item[ii)] Let $ \lambda>0$, and define $ R_0^{\pm}(\lambda) = \lim_{ \epsilon \to 0 } (\sqrt{-\Delta}- (\lambda \pm i \epsilon))^{-1} $. In \cite{ZHZ}, it was shown that $R_0^{\pm}(\lambda)$ maps $L^{2, \sigma} \to \mathcal{H}_{1, -\sigma}$ for $ \sigma >1/2$, see Theorem~4.7. Therefore, for any  $f \in L^{2, \sigma}$ one can define the outgoing solutions as $ u = R_0^{+} (\lambda)f$ using \eqref{R0main}, and incoming solutions as $ u = R_0^{-}(\lambda)f$ using \eqref{R0in}. Moreover, $ R_0^{\pm} (\lambda)f$ are the strong solutions to 
      $$ ( \sqrt{-\Delta} -\lambda) u = f. $$

\item[iii)]  In Lemma~\ref{lem:mapping} the extension of $\chi R_0(z) \chi $ is continuous including the zero energy since $s_k \geq 1$ for all $k$.  In fact,  if we let $r:=|x-y|$, then 
\begin{multline} \label{compactexp}
R_0(z)(r) = \sum_{ j=1}^{\frac{d-1}{2}} \frac{\alpha_{j} z^{j-1} }{r^{ d -j}} + \sum_{ j=1}^{\frac{d-1}{2}} \tilde{c_j} \frac{( iz)^j e^{ izr} \tilde{E}_i(- izr)}{ r^{d-1-j}} \\ + \sum_{ j=1}^{\frac{d-1}{2}} \tilde{c_j} \frac{(- iz)^j e^{- izr} \tilde{E}_i( izr)}{ r^{d-1-j}}  +    \sum_{ j=1}^{\frac{d-1}{2}} c_j \frac{z^j e^{ izr} }{ r^{d-1-j}}
\end{multline}   
  
 and therefore,  as $z \rightarrow 0 $,  one has

   \begin{align} \label{smallz}
 [\chi R_0(z) \chi](x,y) = \chi (x) \frac{\alpha_1}{ |x-y|^{d-1}} \chi(y) + O_{C} (|z|).
   \end{align}
   Here  $O_{C} (|z|^p)$  refers to  compact operator from $ L^2 \to L^2$ with norm bounded by a constant multiple of $ |z|^p$. 
    \end{itemize}
       \end{rmk}
\begin{lemma}  The expansion \eqref{smallz} is valid as $ z \to 0$. 
\end{lemma}   
   \begin{proof}
  If $  |zr| < 1$, one has $\tilde{E}_i(\pm iz r) =  \gamma + \log(\mp i z r) + O(  |z| r ) $, therefore, 
   \begin{align}
   R_0(z)(x,y)= \frac{\alpha_1}{r^{d-1}}  + O \Big( |z| r^{2-d} \Big)
   \end{align} 
   This finishes the proof since as a kernel operator $ V(x)|x-y|^{2-d} \chi(y)$ is a compact operator.

   \end{proof} 
   
 Expansion \eqref{smallz} is important because it suggests that  for potentials having enough decay at infinity, 0 is either a regular point  of $\sqrt{-\Delta} + V$ or it is an eigenvalue with finite geometric multiplicty. Hence, zero is not an accumulation point of the spectrum of $\sqrt{-\Delta} + V$. In the following lemma we define $S_1 L^2$ as projection onto the kernel of $(I+  G_0 V )^{-1}$. Note that $G_0 V$ is a compact operator, therefore, we can define $S_1$, and $S_1$ is a finite rank operator. 
  \begin{lemma} Let $ |V| \les \la x \ra^{-1-}$,  $f \in S_1$ if and only if $ \psi := -G_0Vf \in L^2$ with  $(\sqrt{-\Delta} + V) \psi=0$.
  \end{lemma}
  \begin{proof}
  If $ f \in S_1$ then $( I + G_0V)f = 0 $ , and $f = - G_0V f$. Therefore, 
  $$ (\sqrt{-\Delta} + V) G_0 V f = Vf - Vf = 0.$$
  Moreover, one has 
  \begin{align}
  \frac{1}{|x-y|^{d-1}} = \chi(|x| \gtrsim 2 |y| ) \Big( \frac{1}{ |x|^{d-1}} + O\big( \big(|y| / |x| \big)^{0+}\big) \Big) + \chi(|x| \les 2 |y| ) \frac{1}{ |y|^{d-1}} 
  \end{align}
  Therefore, 
  \begin{multline}
  \psi (x) = \int_{\R^d} \chi(|x| \gtrsim 2 |y| ) \Big( \frac{1}{ |x|^{d-1}}  + O\big( \big(|y| / |x| \big)^{0+}\big) \Big) Vf(y) dy  \\ +  \int_{\R^d} \chi(|x| \les 2 |y| ) \frac{1}{ |y|^{d-1}} Vf(y) dy \les \frac{1}{ \la x \ra^{d-1} } \in L^2 (\R^d) 
  \end{multline}
  provided that $|V(y)| \les \la y \ra^{-1-}$. 
  On the contrary if $(\sqrt{-\Delta} + V) \psi=0$, then $ \psi = - G_0 V \psi $, and $(I + G_0 V) \psi=0$. 

  \end{proof}

  Finally we prove Proposition~\ref{prop:mer}.

\begin{proof}[ Proof of Proposition~\ref{prop:mer} ]

Note that  for any $z$ with $\Im z >0, \Re z > 0$ one has 
\begin{align*}
 R_V(z)-R_0(z) = -R_V(z) V R_0(z) 
\end{align*} 
and $ ( I + VR_0(z) ) R_V(z) = R_0(z)$. 
Therefore, 
\begin{align}
\chi R_V \chi + \chi R_V V  R_0  \chi = \chi R_0 \chi \Rightarrow \chi R_V \chi ( I +  V  R_0  \chi ) = \chi R_0 \chi. 
\end{align} 
Since $(-\Delta)^{1/2}$ is symmetric $( I +  V  R_0  \chi )$ is invertible when $ \Re(z) >0$ and if $\Im (z)$ is large enough by Neumann series and 
\begin{align}
\chi R_V \chi = \chi R_0 \chi ( I+  V R_0   \chi )^{-1}. 
\end{align} 
Moreover, for every fixed $z \in \Lambda$, $ \chi R_0(z) V $  is a compact operator from $L^2$ to $L^2$. Therefore, either  $ (I+ VR_0(z) \chi) $ is invertible or $0$ is an eigenvalue. Hence, $\chi R_V \chi$  continues meromorphically to $\Lambda$  as an operator $L^2$ to $L^2$ by Fredholm alternative. 
 Moreover,  the poles of $\chi R_V(z)\chi$  should satisfy
 $$ 
 (I+V R_0(z) \chi ) \psi = 0 \,\,\,\,\ \Rightarrow  \psi = - V R_0(z) \chi \psi
 $$ 
for some $ \psi \in L_{\text{c}}^2$. 
\end{proof}

\section{ Poles of $ \chi R_V(z) \chi$ }
Recall that the resonances are the poles of the meromorphic continuation of  $\chi R_V(z) \chi$. If $z_0$ is such pole then we have 
$$ \text{mult}(z_0) = \text{rank} \int_{ |z-z_0|=\epsilon } \chi R(z) \chi dz, $$ 
 where $ \epsilon >0$ is small enough so that the closed ball $ \overline{B(z_0,\epsilon)}$  excludes any other pole of $ \chi R(z) \chi$. 
 
 We first want to mention that neither the number of resonances  nor their multiplicities depend on the cut-off function $\chi$. In fact,  for any $ \chi $ that is one in the support of $V$ and vanishes outside of the support, we have 
      $$ \chi R_V(z) \chi = \chi R_0(z) \chi + \chi R_0(z) V R_0(z) \chi + \chi R_0(z) V R_V(z) V R_0(z) \chi. $$ 
     
Since  $\chi R_0(z)\chi$ is holomorphic in $z$ for any odd $d\geq3$,  the number of resonances of $ \chi R_V(z) \chi$ with counting multiplicity is equal to the number of resonances of $V R_V(z) V $ with counting multiplicity.

  As a second note, we would like to mention the well-known relationship between $ \text{mult}(z_0)$ and  the Fredholm determinant of $(I+ V R_0  \chi)$, see e.g. \cite{BS,VD}. Let $K(z)$ be a bounded trace class, and $B(z), C(z)$ be bounded linear operators, and let  $ \text{det} (I-K(z)) = z^{\ell} f(z) $ for some $f(0) \neq 0$ in an open neighborhood $\Omega \subset \mathbb{C}$ of zero, then 
  
    \begin{align} \label{rank} \text{rank} \int_{|z|=\epsilon }  B(z) (I-K(z))^{-1} C(z)  dz  \leq \ell .
 \end{align}
 Note that, by Proposition~\ref{prop:mer}, we have $ V R_V \chi = ( I+  V R_0  \chi )^{-1} V R_0 \chi $. Therefore, one would like to use \eqref{rank} for $K(z) = V R_0(z)  \chi $. However, we will see below that   $( \chi R_0(z) V )$ is not a trace class operator.
 
To solve this problem, we use \eqref{rank} for $ K(z) = ( V R_0(z) \chi )^{d+1}$. The fact that $ ( V R_0(z) \chi )^{d+1}$ is a trace class operator,  see Proposition~\ref{detest} below, and the expansion 
$$  ( I - ( V R_0(z)  \chi )^{d+1}) = \sum_{j=0}^d ( - V R_0(\lambda)  \chi)^j (I+ V R_0(\lambda)  \chi) $$
 will give us that if  $ \text{det} (I-( V R_0(z_0) \chi )^{d+1}) =0 $ and  $\ell(z_0)$ is the algebraic multiplicity of $z_0$, then $\text{mult}(z_0) \leq \ell(z_0)$.
  
 In this section, we let $H(z) :=  \text{det} ( I - ( V R_0(z) \chi )^{d+1}) $, and estimate its order of growth.  
 
 \begin{prop} \label{detest}Let $V(x)$ be as in Theorem~\ref{the:main}. For $ | arg z| \leq a $, and $ \delta <  | z| $,   we have 
\begin{align}
 |H(z)| \les e^{ ( \la z \ra + \log \la a \ra)^d }
\end{align}

\end{prop}

 



We start with the following lemma. 
 \begin{lemma}Let $D^{+}:= \{ z \in \Lambda: \arg z \in (0 , \pi] \} $. Then for any $z_0 \in D^{+}$, we have
   \begin{align} \label{Hs0}
  R_0(z_0) = \sum_{j=1}^{\frac{d-1}{2}} B_j(r, z_0) 
\end{align}
where  $\| B_j(r, z_0) \|_{L^2 \rightarrow H^j} \les \la z_0 \ra^{j-1} $. 
\end{lemma}
\begin{proof}
We first show the statement when $ \arg z_0 \in (0, \pi/2) $, where we originally define the resolvent operator $R_0(z)$.  Recall by \eqref{E1int}, we have for $ \arg z_0 \in (0, \pi/2) $
\begin{align}\label{intrep}  E_i (\pm i \sigma) =  \int_{-\infty}^{\pm i\sigma}  \frac{e^{p}}{p} dp 
\end{align}
 which can be represented as 
\begin{align} \label{nonegb}
  E_i( \pm i \sigma)  = e^{\pm i \sigma} a(\pm \sigma) , \,\,\,\,   
  \end{align}
where $ |\partial^j a(\pm \sigma)| \les |\sigma|^{-1-j}$, $ j =0,1,..., \frac{d-3}{2}$ for $|\sigma | \gtrsim 1$. Using \eqref{nonegb}, we conclude that if $ \arg z_0 \in (0, \pi/2)$, then  
 \begin{align} \label{Eeasy}
  \tilde{E}_i (\pm iz_0s) =  E_i(\pm iz_0s) = h ( |z_0|s) \log |z_0s| + e^{\pm iz_0s} (1-h( |z_0|s) )a(\pm iz_0s) 
\end{align}
where $h(z)$ is a cutoff function supported in  $|z| < 1$. One can compute that for any $s_k \geq 1$ 
\begin{align} \label{Easyder}
&|\partial_s^{s_k} \{ h( |z_0|s) e^{iz_0s} \log |z_0s|\} | \les \sum_{ m+p = s_k}(|z_0|^m s^{-p}) h( |z_0|s)  \les  s^{-s_k}  \\
& |\partial_s^{s_k} \{ ( 1 - h(|z_0|s)) e^{\pm iz_0s} \tilde{E}_i (\pm iz_0s)\} |=  |\partial_s^{s_k} \{( 1 - h(|z_0|s))  a(\pm iz_0s)\} |\les s^{-s_k} .\nn
 \end{align} 
Using \eqref{Easyder} in the expansion \eqref{R0main} for $R_0(z)$,  we obtain 
\begin{align}
\label{R00} |R_0(z_0) (r)| \les \frac{1}{r^{d-1}} + \sum_{k=0}^{\frac{d-3}{2}}  \frac { \la z \ra^{s_k}}{ r^{d-1-s_k}}   
\end{align}
as for $ z $ with $ \Im z \geq 0$ one has $ | \partial_s^{s_k} \{e^{izs}\} | \les \la z \ra^{s_k} $, for any nonnegative real number $s$. That implies \eqref{Hs1} for $ \arg z \in (0, \pi/2)$. 

We next consider when $ \arg z_0 \in (\pi/2, \pi ]$. We show that the extension of  $E_i ( \pm i \sigma)$ at $ z_0 r$ agree with the integral in \eqref{intrep} in this domain. Once we have that, the inequalities in \eqref{Eeasy} and \eqref{Easyder} are valid for $ e^{\pm iz_0s} \tilde{E}_i ( \pm iz_0s)$, and therefore the statement follows by \eqref{R00}.  We show that $E_i(iz_0r)$ agrees with the integral in \eqref{intrep}. The proof for $E_i(-iz_0r)$ follows similarly. 
\begin{align} \label{calc}
\ln (-iz_0r)  + \sum_{k=1}^{\infty}\frac{ (iz_0r)^k}{ k k!} &=  \ln (-iz_0r)  + \int_{0}^{iz_0r} \frac{ e^{p} -1}{p} dp \\
&= \ln (-iz_0r)+  \int_{-1}^{ iz_0r} \frac{e^p -1}{p} dp + \int_{0}^{-1} \frac{e^p -1}{p} dp  \nn  \\
& = \int_{-1}^{iz_0r} \frac{e^p}{p} dp + \int_{-\infty}^{-1}  \frac{e^p}{p} dp - \int_{-\infty}^{-1}  \frac{e^p}{p} dp + \int_{0}^{-1} \frac{e^p -1}{p} dp \nn \\ 
 & = \int_{-\infty}^{iz_0r} \frac{e^p}{p} dp - \gamma = e^{iz_0r} a(i z_0r ) - \gamma \nn
\end{align}
In the last equality we used 
$$   \int_{-\infty}^{-1}  \frac{e^p}{p} dp - \int_{0}^{-1} \frac{e^p -1}{p} dp= \int_{0}^{1} \frac{ 1- e^{-p} - e^{-1/p}}{p} dp =\gamma. $$

Finally, the calculation  \eqref{calc} is also valid for $E_i(iz_0r)$ when $\arg z_0 = \pi/2$. Moreoer, $e^{ iz_0 s}$ has exponential decay when $\arg z_0 = \pi/2$, we have 
\begin{align}\label{negb}
e^{  iz_0 s } E_i( -iz_0 s) = e^{iz_0 s}( \gamma + \log (iz_0 s) + \sum_{n=1}^{\infty}  \frac{(-iz_0 s)^{n} }{n n!} ) = b(z_0s)
\end{align} 
where $ |\partial^j  b(\sigma)| \les |\sigma|^{-1-j}$, $ j =0,1,..., \frac{d-3}{2}$ for $|\sigma| \gtrsim 1$. Using \eqref{negb}, we obtain 
\begin{align} 
 e^{iz_0s} \tilde{E}_i (- iz_0s) =  e^{iz_0s}  h ( |z_0|s) \log |z_0s| + (1-h( |z_0|s) )b(\pm iz_0s). 
\end{align}
Therefore, the bound $|\partial_s^{s_k} \{ e^{  iz_0 s } E_i( -iz_0 s)\} |\les s^{-s_k}$ is valid, and so does \eqref{R00} when $\arg z_0 = \pi/2$. 
\end{proof}
 \begin{lemma}Let $D^{-}:= \{ z \in \Lambda: \arg z \in (\pi, 2\pi] \} $. Then for any $z_0 \in D^{-}$, we have
   \begin{align} \label{Hs1}
  \chi R_0(z_0) \chi = \sum_{j=1}^{\frac{d-1}{2}} \tilde{B}_j(r, -z_0 ) +2 i\Big( \frac{z_0} { 2 \pi } \Big)^{d-1} \mathcal{E}^*_{\chi}(- \bar{z_0})  \mathcal{E}_{\chi}( -z_0) 
\end{align}
where  $\| \tilde{B}_j(r, z) \|_{L^2 \rightarrow H^j} \les \la z \ra^{j-1} $. Here, 
$$ \mathcal{E}_\chi (z):= \int_{ \R^d} e^{ i z x \cdot \omega} \chi(x) f(x) dx.  $$ 
\end{lemma}
\begin{proof}
Let $ z =  e^{-i \pi} z_0$. Then $ \log ( iz_0 s) =  \log (- i zs) + 2 \pi i $, and $ \log(- i z_0s) =  \log (i z s) $. Hence, 
  \begin{align}
e^{-iz_0s} \tilde{E}_i(iz_0s) = e^{i z s} E_i (-i z s), \,\,\,\,\ e^{iz_0s} \tilde{E}_i(-iz_0s) = e^{-izs} E_i(iz s) + e^{-iz s}  2 \pi i 
  \end{align}
 Therefore, 
\begin{align} \label{genR}
R_0(z_0)(r) &= \frac{1}{ ( 2 \pi)^d}    \sum_{k=0}^{\frac{d-3}{2}}  \frac{\tilde{c_k}}{ r^{d-1-s_k}}  \partial_s^{s_k} \Big\{  e^{izs} E_i( -izs) + e^{-izs} E_i( i zs) + 4 \pi i e^{-iz s}   \Big\}\Big|_{(s=|r|)}     \\ 
&= \frac{1}{ ( 2 \pi)^d}  \sum_{k=0}^{\frac{d-3}{2}}  \frac{\tilde{c_k}}{ r^{d-1-s_k}}  \partial_s^{s_k} \Big\{  e^{izs} E_i( -izs) + e^{-izs} E_i( i zs) -4 \pi i  e^{izs} \Big\}\Big|_{(s=|r|)}   \nn  \\ 
&+  \frac{2i}{ (2 \pi)^{d-1}}  \sum_{k=0}^{\frac{d-3}{2}}  \frac{\tilde{c_k}}{ r^{d-1-s_k}}  \partial_s^{s_k} \{ e^{izs}+ e^{-izs}\}\big|_{(s=|r|)} \nn
  \end{align}
 Note that $ \arg z \in (0, \pi]$. Therefore,  the term in the second line in \eqref{genR} holds an estimate in the form of  \eqref{R00}. Moreover, by \eqref{osc}, we have 
\begin{align}
\sum_{k=0}^{\frac{d-3}{2}}  \frac{\tilde{c_k}}{ |x-y|^{d-1-s_k}} \partial_s^{s_k} \{ e^{iz s} + e^{-iz s} \}\Big|_{(s=|x-y|)} =z^{d-1}  \int_{\mathbb S^{d-1}}  e^{i z (x-y) \cdot \omega}  d \omega 
\end{align}
This finishes the proof. 
\end{proof}

We next estimate $ \chi R_0(z) \chi$ for all $ z \in \Lambda$. Let $z_0 \in D^+\cup D^{-} / \{0\}$, and  $z_m := e^{ 2 \pi i m} z_0$, $m \in \mathbb{Z}$.  Then we have 
\begin{align*}
[e^{iz_1s} \tilde{E}_i(-iz_1s) + e^{-iz_1s} \tilde{E}_i(iz_1s)] - [e^{iz_0s} \tilde{E}_i(iz_0s) + e^{-iz_0s} \tilde{E}_i(-iz_0s)] =  2 \pi i  [ e^{iz_0s} + e^{-iz_0s} ] 
\end{align*}
pointwise. Therefore, by \eqref{osc}, \eqref{Hs0} and \eqref{Hs1} we obtain
\begin{align*}
R_0(z_1) (x,y) - R_0(z_0) (x,y) =  2 \pi i  \widehat{ \sigma_{z_0 \mathbb{S}^{d-1}}} (r)
\end{align*}
as an operator $L^2_c \to L^2_{\text{loc}}$. 
More generally, we have 
\begin{align} \label{zm}
R_0(z_m)(x,y) - R_0(z_0)(x,y)  &=   2 \pi i  m\widehat{ \sigma_{ z_0 \mathbb{S}^{d-1}}} ( r)
\end{align}
  Equality \eqref{zm} results in the following corollary. The bound $ \eqref{Hexp}$ will follow from \eqref{genR}, \eqref{Eeasy}, and \eqref{zm}. 

 \begin{corollary} Let $z_m = e^{2 \pi i m} z_0$. Then one has for $ z_0 \in D^{+}$
  \begin{align} 
 \chi R_0(z_m) \chi = \sum_{j=1}^{\frac{d-1}{2}} B_j(r, z_0 ) +  i m \Big( \frac{z_0} { 2 \pi } \Big)^{d-1}   \mathcal{E}^*_{\chi}( \bar{z_0} ) \mathcal{E}_{\chi}( z_0)  \label{Hs} 
\end{align}
and for $ z_0 \in D^{-}$
   \begin{align}  \label{Hss} 
 \chi R_0(z_m) \chi = \sum_{j=1}^{\frac{d-1}{2}} \tilde{B}_j(r, -z_0 )  + 2 i\Big( \frac{z_0} { 2 \pi } \Big)^{d-1} \mathcal{E}^*_{\chi}( i\bar{z_0})  \mathcal{E}_{\chi}( -i\bar{z_0}) +   i m \Big( \frac{z_0} { 2 \pi } \Big)^{d-1}   \mathcal{E}^*_{\chi}(z_0 ) \mathcal{E}_{\chi}( \bar{z_0} )
\end{align}
where $\| \tilde{B}_j(r, z_0) \|_{L^2 \rightarrow H^j}, \| B_j(r, z_0) \|_{L^2 \rightarrow H^j} \les \la z \ra^{j-1}$. Moreover, we have 
\begin{align} \label{Hexp}
\| \chi R_0(z_m) \chi \|_{ L^2 \rightarrow H^1} \les \la m \ra e^{ \la z_0 \ra}, 
\end{align} 
\end{corollary}

We are now ready to prove Proposition~\ref{detest}.

\begin{proof}[Proof of Proposition~\ref{detest}]
To prove the statement we first show that $(V R_0 \chi)^{d+1}$ is in trace class. Once we show that, we estimate  $ |H(\lambda)|$ using  
$$ |H(\lambda)| \leq \Pi_{j} s_j (( V R_0 \chi)^{d+1}) $$ 
for $ \lambda \geq 1$, where $s_j$'s are the singular values of $ ( V R_0 \chi)^{d+1}$. Let $A,B$ compact and $C : H_1 \rightarrow H_2$ be bounded operator, then we have the following two properties for $k,j \geq 0$, see Theorem~1.6 and  Theorem~1.7 in \cite{BS}.
\begin{enumerate}
\item[i)] $s_{k+j+1} ( A+B) \leq s_{k+1}(A) + s_{j+1}(B) $
\item[ii)] $s_j(AC) \leq s_j(A) \| C\|_{ H_1 \rightarrow H_2}$
\end{enumerate} 
We start proving that $(V R_0 \chi)^{d+1}$ is in trace class. We first note that $(V R_0 \chi)^{d+1}$ is compactly supported in a ball of $ \overline {B(0, R)}$ for some $0<R<\infty$. Therefore, we have for any $ \lambda = z e^{ 2m\pi i}$ such that  $ \arg z \in (0, 2\pi]$
\begin{multline}
s_j(   VR_0(\lambda)  \chi ) \leq \| V\|_{L^\infty} s_j ( ( -\Delta +I)_{{B}}^{-\f 12}  ( - \Delta + I)^{\f12}_{B} \chi R_0 \chi ) \\ \les \| V\|_{L^\infty} s_j ( ( -\Delta + I)_{B}^{-\f 12}) \| \chi R_0 \chi \|_{ L^2 \rightarrow H^1} 
\les \| V\|_{L^\infty}j^{-1/d} \la m  \ra e^{|\lambda| }
\end{multline} 
where  $ \Delta_B$ be the Laplace-Beltrami operator in $ \overline {B(0, R)}$. In the second inequality we used  property $\text{ii})$, and in the last equality we used  Weyl law for eigenvalue asymptotics and \eqref{Hexp}. Therefore, 
\begin{align}
Tr ( (VR_0(\lambda)  \chi)^{d+1}) \leq \sum_{j \in \mathbb{Z}^{+}} s_j( (VR_0(\lambda)  \chi)^{d+1}) \les \la m \ra  e^{|\lambda| } \label{trbound}
\end{align}
and hence, $(\chi R_0 V )^{d+1}$ is trace class. 

We next estimate $|H(\lambda)|$. Recall \eqref{Hs}, we have for any $ z \in D^{+}$ 
\begin{align} \label{differ}
VR_0(z e^{2 \pi i m} ) \chi = \sum_{k=1}^{\frac{d-1}{2}} V B_k(r, z )+  2 \pi  m i z^{d-1}  V  \mathcal{E}^*_{\chi}(\bar{z} ) \mathcal{E}_{\chi}( z) 
 \end{align}
Therefore, by the property $\text{i})$ 
\begin{align}\label{tobeopt}
s_j (V R_0(z e^{2\pi i m } ) \chi)  \les \| V\|_{L^{\infty}} \Big( \sum_{k=1}^{\frac{d-1}{2}} s_{(\frac{j}{ d-1})} ( B_k(r, z)) + s_{(\frac{j}{2})}  ( m  z^{d-1}   \mathcal{E}^*_{\chi}( \bar{z}) \mathcal{E}_{\chi}( z))  \Big)
 \end{align}
 Similarly, by \eqref{Hss} one has for any $z \in D^{-}$ 
\begin{align}\label{optsecond}
 s_j (V &R_0(z e^{2\pi i m } )\chi)  \les \\
&    \| V\|_{L^{\infty}} \Big( \sum_{k=1}^{\frac{d-1}{2}} s_{(\frac{j}{ d-1})} ( B_k(r, z)) + s_{(\frac{j}{4})}  ( m  z^{d-1}   \mathcal{E}^*_{\chi}( \bar{z}) \mathcal{E}_{\chi}(z) ) + s_{(\frac{j}{4})}  ( z^{d-1}   \mathcal{E}^*_{\chi}( -\bar{z}) \mathcal{E}_{\chi}(-z ))    \Big) \nn 
 \end{align}
First note that, by \eqref{Hs} we have 
\begin{align} \label{B}
s_{j} ( B_k(r, z)) \les ( ( -\Delta +I)_{B}^{-\f k 2}) \| B_k(r, z )\|_{L^2 \to H^k} \les  j^{-k/d}\la z \ra^{k-1}. 
\end{align}
This estimates the first term in \eqref{tobeopt} and \eqref{optsecond} by $ \sum_{k=1}^{\frac{d-1}{2}} j^{-k/d}\la z \ra^{k-1}$. 

We next estimate the second term in \eqref{tobeopt}. The idea that we use here has been first presented in \cite{Vod2}. Recall that 
$$ \mathcal{E}^*_\chi (z)= \int_{ \mathbb S^{d-1}} \chi(x)e^{ i z x \cdot \omega} g(\omega) d\omega.  $$  
Therefore, 
$$ 
s_{j/2}  ( m  z^{d-1}   \mathcal{E}^*_{\chi}( \bar{z}) \mathcal{E}_{\chi}( z)) \leq C_0 e^{|z|} s_{j/2} ( (- \Delta + I)_{\mathbb S^2} ^{-l} ) \| (- \Delta + I)_{\mathbb S^{d-1}}^ l \la m \ra \mathcal{E}_\chi (z) ) \|_{L^2 (\R^n) \to L^2 (\mathbb S^{2})}
$$ 

We also have, 
\begin{align}
 \| (- \Delta + I)_{\mathbb S^2}^ l \la m \ra  \mathcal{E}_\chi (z) \|_{L^2 (\R^n) \to L^2 (\mathbb S^{2})} \leq e^{ C_2 ( \la z \ra + \log \la m \ra) }  (2 l) !
\end{align}
This gives 
\begin{align*}
s_{j/2}  ( m  z^{d-1}   \mathcal{E}^*_{\chi}( \bar{z}) \mathcal{E}_{\chi}( z))
 \leq C_1 j^{-l} e^{ C_2 ( \la z \ra + \log \la m \ra) }  (2 l) !
\end{align*}
Pick $ j/ l = C_3 e $, then $ j^{-l} ( 2 l) ! \leq  e^{-C_4 j^{\f1{d-1}}}$. Hence, 
\begin{align} \label{optimize}
s_{j/2}  ( m  z^{d-1}   \mathcal{E}^*_{\chi}( \bar{z}) \mathcal{E}_{\chi}( z)) \les  e^{ C_1 (\la z \ra + \log \la m \ra) - C_2 j^{\f1{d-1}}} 
\end{align}
for some $C_1, C_2 \in \R$.

Using a similar argument that leads to \eqref{optimize}, we obtain for any $ z \in D^{-}$ 
 \begin{align} \label{C}
  & s_{(\frac{j}{4})}  ( m  z^{d-1}   \mathcal{E}^*_{\chi}( \bar{z}) \mathcal{E}_{\chi}(z) )  \les  e^{ C_3 (\la z \ra + \log \la m \ra) - C_4 j^{\f1{d-1}}}  \\
  &  s_{(\frac{j}{4})}  ( \bar{z}^{d-1}   \mathcal{E}^*_{\chi}( z) \mathcal{E}_{\chi}( \bar{z}))  \les e^{ C_5 \la z \ra  - C_6 j^{\f1{d-1}}} \nn
 \end{align} 

 Using \eqref{B}, \eqref{optimize}  in \eqref{tobeopt}, and \eqref{B},\eqref{C} in  \eqref{optsecond}, we obtain 
\begin{align}
s_j ( V R_0(z e^{i m \pi}) \chi ) \les  \begin{cases} 
  \sum^{(d-1)/2}_k \la z \ra^{k-1} j^{-k/d} &  j \geq \tilde{C}( \la z \ra + \log \la m \ra)^{d-1} \\
      e^{\la z \ra + \log \la m \ra}  & j \leq \tilde{C} ( \la z \ra + \log \la m \ra)^{d-1}
   \end{cases}
\end{align} 
for some $\tilde{C} \in \R$ determined according to $C_i$, $i=1,2,...,6$.

Note that if $  j \geq \tilde{C}( \la z \ra + \log \la m \ra)^{d-1} $ then $ z^{k-1} j^{- (k-1)/d} \les \la z\ra^{(k-1)/d} $. Also since,  $ k \leq \frac{d-1}{2}$, we further have $ z^{k-1} j^{- (k-1)/d} \les \la z\ra^{(d-3)/2d}$, for $ d \geq 3$.   Therefore, 
\begin{align}\label{finalest}
s_j ((V R_0(z e^{ 2 i m  \pi})  \chi)^{d+1}) \les  \begin{cases} 
 \la z \ra^{\frac{d-3}{2d} (d+1)}  j^{- \f {d+1}{d}} &  j \geq \tilde{C} ( \la z \ra + \log \la m  \ra)^{d-1} \\
      e^{\la z \ra + \log \la m  \ra}  & j \leq \tilde{C} \la ( \la z \ra + \log \la m \ra)^{d-1}.
   \end{cases}
\end{align} 
Recall, we have $ |H(\lambda)| \leq \Pi_{j} s_j (( V R_0 \chi)^{d+1}) $. Therefore, using \eqref{finalest} for $ \lambda = z e^{2i \pi m} $  we obtain 
\begin{align} \label{thirdpi}
|H(\lambda)| \les \Big( \prod_{ j \leq \tilde{C} (\la  z \ra  + \log \la m  \ra)^{d-1}} e^{ \la z \ra + \log \la m  \ra} \Big) \Big( e^{ \la z \ra^{\frac{d-3}{2d} (d+1)}  \sum_{j \geq \tilde{C} \la z \ra + \log \la m \ra} j^{-\f {d+1} {d}} \Big)} \\ \les e^ {(\la  z  \ra  + \log \la m  \ra)^d } . \nn
\end{align}

This finishes the proof of Proposition~\ref{detest}. 
\end{proof}
\section{Proof of Theorem~\ref{the:main}}
In the proof of Theorem~\ref{the:main}, we use Theorem~\ref{th:vodev} from \cite{VD}.
\begin{theorem} \label{th:vodev}Let $h(z)$ be a holomorphic function on $\{ z \in \Lambda: |z| \leq r \}$ for some $r \geq1$, and 

$$ h(z) = 1+ O_1 (|z|^{\gamma}) \,\,\,\ \text{as} \,\,\,\,\ z \rightarrow 0, |arg z | \leq a. $$

$ \forall a \geq 1$, with $ \gamma > 0$ independent of $a$. Then for any integer $m \geq 2 /\gamma$ the number of $ \tilde{N} ( r/2, \pi m /4 ) $ of the zeros of $h(z)$ in $\{ z \in \Lambda: |z| \leq r/2 , |argz| \leq \pi m/4\} $ repeated according to multiplicity, satisfies the bound  with a constant $C>0$ independent of $h(z)$, $r$ and $m$ , where $ \log^{+} x = \max \{ \log x, 0\}$. 
\begin{align} \label{vodbound}
 \tilde{N} ( r/2, \pi m /4 )\leq &\ C m \sup_{ |\theta| \leq \pi m/2} \log^{+} |h ( re^{i\theta})| \\
  &+ C r^{1/m} \int_0^r t^{-1-1/m} \log^{+} | h( t e^{-i\pi m/2} )h( t e^{ i\pi m/2} )| dt  \nn
\end{align} 
\end{theorem}

\begin{proof}[Proof of Theorem~\ref{the:main}] 
Note that by \eqref{compactexp}, we have 
$$(V R_0(z) \chi) = \sum_{j=0}^{\infty} z^j K_{1,j} + \log(z) \sum_{j=0}^{\infty} z^{2j} K_{2,j} $$
for $ |z| \ll 1$ where $K_{1,j}$ and $K_{2,j}$ are trace class operators, therefore, we have 
$$ H(z) =  \sum_{j=0}^{\infty}c_{1,j}  z^j  + \log(z) \sum_{j=1}^{\infty} c_{2,j} z^{2j},\,\,\,\,\,\,\,\, |z| \ll 1  $$ 
Note that since $0$ is non-degenerate there is at least one $ c_{1,j}, c_{2,j}  $ that is not zero. Let $p$ be the smallest integer such that one of the $ c_{1,p}, c_{2,p/2}  $ is not zero. Then, define 

\begin{align}
H_1(z):=  \begin{cases} 
\frac{ H(z)}{ c_{1,p} z^{p}}  & \text{if} \,\, c_{1,p} \neq0  \\
      \frac{ H(z)}{ c_{2,p} z^{p} \log (z) } & \text{if}  \,\,\ c_{2,p/2}\neq 0\\
      \frac{ H(z)}{  z^{p} (c_{1,p} +c_{2,p} \log (z) ) } &  \text{if} \,\,\ c_{1,p} , c_{2,p/2} \neq 0 \\
        \end{cases}
\end{align} 

 Note that the estimate in  Proposition~\ref{detest} is still valid for $H_1(z)$. Moreover, $ \log^{+} | H_1(z)| \les z^{1-}$ for $z \ll 1$. So,  we can use \eqref{vodbound} for $H_1(z)$ to estimate the number of zeros of $H_1$. 
 
 We have 
 \begin{align*}
 m \sup_{ |\theta| \leq \pi m/2} \log^{+} |H_1 ( re^{i\theta})| \les m  ( \la r \ra + \log \la m \ra)^d . 
 \end{align*}
Furthermore, 
\begin{multline}
 r^{1/m}  \int_0^r t^{-1-1/m} \log^{+} | H ( t e^{-i\pi m/2} )H ( t e^{ i\pi m/2} )| dt  \\ 
   \les r^{1/m} \int_0^r t^{-1/m-} dt + \int_{r_0}^{r} t^{-1-1/m} (  \la t\ra + \log \la m \ra )^d dt  \les r^{d} + ( \log \la m \ra )^d+  r^{1/m} ( \log \la m \ra )^d. 
\end{multline} 
Note that, using $ ab \les a^k + b^{k/k-1}$ for $k\geq 2$, and $ a,b >0$, we have 
$$ r^{1/m} ( \log \la m \ra )^d \les r^{d}+ ( \log \la m \ra )^{d ( \frac{d}{d-m})} \les  r^{d} + m  ( \log \la m \ra )^d $$ 
Therefore, the zeros of $H_1(z)$ is established as ($m \geq 1$) 
$$  \tilde{N} ( r/2, \pi m /4 )\les m  ( \la r \ra + \log \la m \ra)^d$$
Including the zero energy, we obtain Theorem~\ref{the:main}. 
\end{proof}

\appendix \label{app}
\section{}
In this section we compute the scattering matrix for $H$ following \cite{DZ}.  We first note that as in the case of the Schr\"odinger, all solutions to  $ (\sqrt{-\Delta} - \lambda )  \psi_0=0$ for $ \lambda > 0$,  are the superposition of the elementary plane wave solutions $e^{ i \lambda \la x , \omega \ra }$  where $ x \in \R^d$, and $ \omega \in \mathbb{S}^{d-1}$, i.e. 
\begin{align} \label{psi0def}
\psi_0(x) = \int_{\mathbb{S}^{d-1}} e^{ i \lambda \la x , \omega \ra } \phi(\omega) d \omega.
\end{align}
Moreover, since any solution to $ (\sqrt{-\Delta} - \lambda ) \psi=0$  is also solution to $(-\Delta - \lambda^2  ) u=0$, $\psi_0(x)$ is the unique solution to  $ (\sqrt{-\Delta} -\lambda )  \psi_0=0$ such that as $ |x| \rightarrow \infty $, 
\begin{align} \label{psi0}
\psi_0(x) =   \frac{c_d  }{ ( \lambda r )^{\frac{d-1}{2} }} ( e^{i \lambda r}  \phi(\theta) + e^{-i \lambda r} i^{1-d}   \phi(-\theta) ) + O( r^{-\frac{d+1}{2}}). 
\end{align}
Here, we let $x = r \theta$, and used the following expansion in the weak sense as $ r \to \infty$ 
\begin{align}
e^{ i \lambda \la x , \omega \ra } =  \frac{1}{ ( \lambda r )^{\frac{d-1}{2} }}\Big(    \frac{( 2 \pi)^{\frac{1}{2} (n-1)}}{e^{\frac{1}{4} \pi (n-1)i}}  e^{i \lambda r}  \delta_\omega(\theta) + e^{\frac{1}{4} \pi (n-1)i} ( 2 \pi) ^{\frac{1}{2} (n-1)} e^{i \lambda r} \delta_{-\omega}(\theta) +  O(r^{-1})\Big). 
\end{align}

 Note that if one uses  $e^{- i \lambda \la x , \omega \ra }$ instead of $e^{i \lambda \la x , \omega \ra }$ in \eqref{psi0def}, then one obtains a similar expression to \eqref{psi0} with different parametrization. The relation between these two parametrization is given by the  absolute scattering matrix. In particular, the operator that maps $\phi(\theta)$ to $i^{1-d}   \phi(-\theta)$ is the absolute scattering matrix for the free equation, and we represent it as $S_{\text{abs},0}(\lambda)$.

To define the same operator for $H$, we we look for a generalized eigenfunctions  for $H$ which behave like plane waves. We define these family of functions  in the form of 
\begin{align}
\psi(x, \lambda) := \int_{\mathbb{S}^{d-1}} w(\lambda, \omega, x) \phi(\omega) d \omega 
\end{align}
where $ w(\lambda, \omega, x) = e^{ -i \lambda \la x , \omega \ra } + u(x, \lambda, \omega) $, and $ u \to 0 $ pointwise as $ x \to \infty$. In particular, we would like to solve $u$ from the following equation. 
\begin{align} \label{Rvsolution}
( \sqrt{-\Delta} + V - \lambda)u = - V  e^{ -i \lambda \la x , \omega \ra }
\end{align}

Below we show that if $H$ has no positive eigenvalues then $R_V(\lambda)$ is defined for all $ \lambda >0$, and $u= - R_V(\lambda) ( V e^{- i \lambda \la \cdot, \omega }\ra )$ solves \eqref{Rvsolution}.  Indeed, it is the unique outgoing solution. 

 We start with the following two lemmas. 

\begin{lemma} \label{noout} Let $ V$ be compactly supported, real valued and  bounded potential. If  $u= R_0(\lambda) g$  for some $ g \in L^\infty_c$ solves $(H-\lambda)u=0$ then $\lambda >0 $ is an eigenvalue.

\end{lemma}
\begin{proof} 
Note that, if $ \lambda >0$ one has 
\begin{align}
R_0( \lambda ) = 2 \lambda R_{\Delta} (\lambda) +  R_0(-\lambda) 
\end{align}
where $R_{\Delta} (\lambda)$ is the outgoing resolvent operator of the free Schr\"odinger operator. Moreover, for $ x = r \theta $, one has 
\begin{align}\label{R0exp}
 R_{\Delta} (\lambda)f = e^{i \lambda r} r^{\frac{1-d}{2}} h(\theta ) + O_{L^2(\R^d)},\,\,\,\, h( \theta) = c_n  \lambda^{\frac{d-3}{2}} \widehat{f}( \lambda \theta), 
\end{align}
Therefore, using \eqref{R00} and \eqref{calc} we have for $ \lambda>0$ 
\begin{align} \label{Vtheta}
R_0(\lambda)f = e^{i \lambda r} r^{\frac{1-d}{2}} h(\theta ) + O_{L^2(\R^d)},\,\,\,\, h( \theta) = c_n  \lambda^{\frac{d-1}{2}} \widehat{f}( \lambda \theta), 
\end{align}
 Now, let $u = R_0(\lambda)g$ be  solution to $(H-\lambda)u=0$ then by Remark~\ref{rmk} part $\text{ii)}$ one has  $ (I + V R_0(\lambda) )g=0$, and $g = -Vu$. Hence, it is enough to show that $\widehat{Vu}( \lambda \theta)=0$ to establish the proof. 
 
To see $\widehat{Vu}( \lambda \theta)=0$, we note that one has for $ f \in \mathcal{S}$, $ \lambda >0$
\begin{align} \label{trace}
 \lim_{\epsilon \to 0} \la R_0(\lambda + i \epsilon)f , f \ra & =  \lim_{ \epsilon \to 0}\int_{\R^d} \frac{ |\hat{f}(\xi)|^2}{ |\xi| -(\lambda +i \epsilon)}  d\xi \nn \\ 
& = \text{p.v.}\int_{ 0}^{\infty}  \int_{\mathbb{S}^{d-1}} \frac{ |\hat{f} ( r \theta)|^2 r^{d-1} } {(r-\lambda)} dr d \omega + i \pi \lambda^{d-1} \int_{\mathbb{S}^{d-1}} |\hat{f} ( \lambda \theta)|^2 d \omega. 
\end{align}  
Moreover, 
\begin{align*}
0&= \lim_{\epsilon \to 0} \la R_0(\lambda + i \epsilon)g, ( I + R_0(\lambda + i \epsilon)g \ra  \\ 
 &= \lim_{\epsilon \to 0} \la R_0(\lambda + i \epsilon)g, g \ra + \| u\|^2_{L^2}
\end{align*}
Note that $g=-Vu$ is compactly supported, bounded and $V$ is real. Therefore, $\widehat{Vu}( \lambda \theta) = 0$ almost everywhere. This establishes the statement. 

\end{proof}

\begin{rmk} \label{rmk2}
Let $u \in \mathcal{H}_{1, \text{loc}}$ and $(H - \lambda)u= 0$ for $ \lambda >0$, and let $ \hat{u}$ be the Fourier transform of $u$.  If  $ \hat{u} \in L^1_{\text{loc}}$ then $u \in L^{2, \sigma}$ for $ \sigma >1/2$. This is a consequence of $\widehat{Vu}( \lambda \theta)=0$ as one has 
\begin{align}
\hat{u}(\xi) = \frac{ \widehat{Vu} (\xi)} {|\xi|- \lambda} = \chi(|\xi| < \lambda /2) \frac{ \widehat{Vu} (\xi)} {|\xi|- \lambda } +\chi(|\xi|  \geq \lambda /2)  \frac{ \widehat{Vu} (\xi)} {|\xi |-\lambda}
\end{align}
and by Theorem~B.1 in \cite{Agmon} the second term maps $L^2 \to L^{2, \sigma}$. Note that the first term is in $\mathcal{H}_s$ for any $ s < \frac{d}{2} +1$. 
\end{rmk}

Recall that we originally defined $R_0(z)$, and $ R_V(z)$ for $\Im z > 0, \Re z >0$. Moreover, by Remark~\ref{rmk} part $\text{ii)}$ we have $\lim_{\epsilon \to 0} R_0(\lambda+i \epsilon)$  exist as an operator $L^{2,\sigma} \to L^{2,-\sigma}$ and therefore $u=R_0(\lambda)f$ for $f \in L^{2,\sigma}$ solves $ (\sqrt{-\Delta} - \lambda) u=f$. In the following lemma,  we show that $R_V(\lambda)f $  solves $(H-\lambda)u=f$ for $ \lambda \in \R^{+} \backslash e_+ (H)$ and $ f \in L^{2, \sigma}$.

\begin{lemma}\label{solves} Let $ V$ be compactly supported real, bounded potential. Moreover, let $ e_{+} (H)$ be the discrete set of positive eigenvalues of $H$. Then for any $ f \in L^{2,\sigma}$, and $ \lambda \in \R^{+} \backslash e_{+} (H) $,  $ R_V (\lambda) f \in \mathcal{H}_{1,-s}$  and it solves 
   $$  ( \sqrt{-\Delta} + V - \lambda) u = f$$ 

\end{lemma}

\begin{proof} To prove the statement it is enough to show that $(I+ R_0(z) V)^{-1}$ exist in the uniform topology as an operator $\mathcal{H}_{1,-\sigma} \to \mathcal{H}_{1,-\sigma}$ if and only if $ \lambda \in \R^{+} \backslash e_{+} (H)$. If $(I+ R_0(z) V)^{-1}$ exist one has for any $ \lambda \in \R^{+} \backslash e_+ (H)$, $ R_V^{+} (\lambda) f \in \mathcal{H}_{1,-s}$  and 
  $$ ( H- \lambda) R_V(\lambda)  f = (\sqrt{-\Delta} -\lambda) (I+T(\lambda)) (I+ T(\lambda))^{-1} R_0(\lambda)f = f. $$

  We first consider when  $z \notin \R^{+}$. Note that, by resolvent identities we have for any $ f \in L^2$ 
\begin{align}
R_V(z)f+ R_0(z) V R(z)f = R_0(z) f \Rightarrow (I+ T(z)) u = R_0(z) f 
\end{align}
where $u= R_V(z)f \in \mathcal{H}_{1}$. Note that this implies that $\overline{R(I+T(z))} = \mathcal{H}_{1,-\sigma}$. Therefore, by Fredholm alternative $(I+ T(z))^{-1}$ exist for all $\Im z > 0, \Re z > 0$. Moreover, if $\lambda \in \R^{+}$, again by Fredholm alternative $(I+T(\lambda)u=0$  if and only if  $u = - R_0 (\lambda) (Vu)$, and by Lemma~\ref{noout}, $u$ is an eigenvalue. Therefore $ \lambda \in e_{+} (H)$. 

Now, suppose $ \lambda \in e_{+} (H)$, then there exist $u \in D(H)$ such that $(H-\lambda)u=0$, then $( \sqrt{-\Delta} - z)u + Vu = (z-\lambda)u$, and therefore, $u + R_0(z) Vu = (z-\lambda) R_0(z)u$. Letting $z \to \lambda$, by Remark~\ref{rmk2} we have $(I+R_0(\lambda)V)u=0$.

\end{proof}

Next, note that by Lemma~\ref{solves}, we have 
\begin{align*}
 w &=  e^{- i \lambda \la x , \omega \ra } + u(x, \lambda, \omega)  \\
    & = e^{ -i \lambda \la x , \omega \ra } - R_V(\lambda) ( V e^{ i \lambda \la \cdot, \omega }\ra ) = e^{ -i \lambda \la x , \omega \ra } - R_0(\lambda) (I+T(\lambda))^{-1} ( V e^{ i \lambda \la \cdot, \omega }\ra )
  \end{align*} 
Moreover, using \eqref{R0exp}
\begin{multline}
\int_{\mathbb{S}^{d-1}} R_0(\lambda) ( I+ V R_0(\lambda))^{-1} ( V e^{- i \lambda \la \cdot, \omega }\ra ) \phi(\omega) d \omega = \\
  = \frac{\lambda^{\frac{d-1}{2}} }{ r^{\frac{d-1}{2}}} e^{i \lambda r} \int_{\mathbb{S}^{d-1}}  \int_{\R^d} e^{- i \lambda \la \theta, x \ra} (I + VR_0(\lambda))^{-1} V e^{-i\lambda \la \cdot , \omega \ra} \phi(\omega) d \omega d x + O_{L^2} \\ 
  = \frac{\lambda^{\frac{d-1}{2}} }{ r^{\frac{d-1}{2}}} e^{i \lambda r}  \mathcal{E}(\lambda) (I + VR_0(\lambda))^{-1} V\mathcal{E^*}(\bar{\lambda})\phi + O_{L^2}. 
  \end{multline}
Therefore, 
\begin{multline}
\psi(x)= 
\frac{c_d }{ ( \lambda r )^{\frac{d-1}{2} }}\Big( e^{-i \lambda r}  \phi(\theta) + e^{i \lambda r} i^{1-d}  \Big( \phi(-\theta) + \tilde{a}_d \lambda^{d-1}  \mathcal{E}(\lambda) (I + VR_0(\lambda))^{-1} V\mathcal{E^*}(\bar{\lambda})\phi \Big) + O_{L^2} \nn
\end{multline}
The absolute scattering matrix is now defined as the map that sends $ \phi(\theta)$ to $i^{1-d}  \Big( \phi(-\theta) + \tilde{a}_d \lambda^{d-1}  \mathcal{E}(\lambda) (I + VR_0(\lambda))^{-1} V\mathcal{E^*}(\bar{\lambda})\phi \Big) $, and we represent it as $S_{\text{abs}}(\lambda)$. 
Finally, the scattering matrix arises as 
 \begin{align} \label{scmatrix}
 S(\lambda) = S_{\text{abs}}(\lambda) S^{-1}_{\text{abs},0}(\lambda) = I + a_d \mathcal{E}(\lambda) (I + VR_0(\lambda))^{-1} V\mathcal{E^*} (\bar{\lambda}) : L^2(\mathbb{S}^{d-1})  \to L^2(\mathbb{S}^{d-1})
  \end{align}
We finish this section with the following theorem which is a consequence of \eqref{scmatrix}. Here, we have uniqueness only because we assume the absence of positive eigenvalues. 
\begin{theorem} Suppose that $V$ is real, bounded and compactly supported, and $ \lambda>0$. Moreover, let $ H$ has no positive eigenvalues. Then for any $ g \in C^{\infty}(S^{d-1})$ there exist unique $ f \in C^{\infty}(S^{d-1})$ and $ \psi \in H^{1}_{\text{loc}}(\R^d)$ such that 
$$
(H- \lambda) \psi=0,\,\,\,\,\, \psi(r \theta) = r^{-{\frac{d-1}{2}}} \Big( e^{i \lambda r} f(\theta) + e^{-i \lambda r} g(\theta) \Big) + O ( r^{- {\frac{d+1}{2}}})
$$
\begin{proof}
We only prove the uniqueness. We will show that if $ \psi_0(r \theta)= r^{-{\frac{d-1}{2}}} e^{i \lambda r} f(\theta) + O ( r^{- {\frac{d+1}{2}}}) $ solves $(H- \lambda) \psi=0$, then $ \psi_0(x) = -R_0(\lambda) V \psi_0(x)$. Lemma~\eqref{noout} then establishes the uniqueness. 

Let $\psi_0(x) = -R_0(\lambda) V \psi_0(x) + h$ where $h \in L^{2,-\sigma}$ solves $(H- \lambda) \psi=0$. Then, $ (\sqrt{-\Delta} - \lambda ) h=0$, and by \eqref{psi0} 
\begin{align*}
h =   \frac{c_d  }{ ( \lambda r )^{\frac{d-1}{2} }} ( e^{i \lambda r}  \phi(\theta) + e^{-i \lambda r} i^{1-d}   \phi(-\theta) ) + O( r^{-\frac{d+1}{2}}). 
\end{align*}
for some $\phi$.  Note that we assume $ \psi_0(r \theta)$ to be outgoing. Therefore, $ \phi =0$ resulting in $h=0$. 
\end{proof}

\end{theorem}
\section*{Acknowledgement}
The author would like to thank John Schotland for suggesting the  study of the resonances of half Laplace operator.

\end{document}